\documentclass[a4paper]{article}
\usepackage[margin=1.0in]{geometry}

\usepackage[l2tabu,orthodox]{nag}
\usepackage{indentfirst}
\usepackage{amssymb,amsfonts}
\usepackage[]{mathtools}
\usepackage{cmap}
\usepackage[T2A]{fontenc}
\usepackage[utf8]{inputenc}
\usepackage{ucs}
\usepackage[russian,english]{babel}
\usepackage[babel = true]{microtype}
\usepackage{graphicx}
\usepackage[colorinlistoftodos, textsize=tiny]{todonotes}

\usepackage{color}
\definecolor{darkblue}{rgb}{0,0,.75}
\definecolor{darkred}{rgb}{.7,0,0}
\definecolor{darkgreen}{rgb}{0,.7,0}

\usepackage[
    draft = false,
    unicode = true,
    colorlinks = true,
    allcolors = blue,
    hyperfootnotes = true
]{hyperref}
\usepackage{amsmath}
\usepackage{amsthm}
    \theoremstyle{plain}
    \newtheorem{theorem}{Theorem}[section]
    \newtheorem{lemma}{Lemma}[section]
    \newtheorem{proposition}{Proposition}[section]
    \newtheorem{corollary}{Corollary}[section]
    \theoremstyle{definition}
    \newtheorem{definition}{Definition}
    \newtheorem*{notation}{Notation}
    \newtheorem*{example}{Example}

\usepackage{thmtools}    

\usepackage{hyperref}
\usepackage[nameinlink]{cleveref}

\makeatletter
\newcommand*{\indep}{%
  \mathbin{%
    \mathpalette{\@indep}{}%
  }%
}
\newcommand*{\@indep}[2]{%
  \sbox0{$#1\perp\m@th$}%        box 0 contains \perp symbol
  \sbox2{$#1=$}%                 box 2 for the height of =
  \sbox4{$#1\vcenter{}$}%        box 4 for the height of the math axis
  \rlap{\copy0}%                 first \perp
  \dimen@=\dimexpr\ht2-\ht4-.2pt\relax
  \kern\dimen@
  {#2}%
  \kern\dimen@
  \copy0 %                       second \perp
} 
\newcommand{\SymPol}[2]{\overset{o}{P}{}^{#1}_{#2}}
\usepackage{bbm}
\title{Generalized Chacon polynomial constructions}

\author{Vladislav Slyusarev}

\date{\today}

\begin{document}
\maketitle

\section{Introduction}
The present work is devoted to Chacon’s automorphism, which was first described by Friedman in \cite{friedman}. This rank-one automorphism is well-known for a number of unusual ergodic and
spectral properties, hence it is used in the construction of examples in the theory of dynamical systems. The classical version of the automorphism if built in the 'cutting-and-stacking' procedure.

We build the transformation inductively. As a base, we have a Rokhlin tower of height $h_0 := 1$, which we will call Tower 0. At the $n$-th step, Tower $n-1$ (of height $h_{n-1}$) is divided into $3$ equal sub-columns and a spacer is inserted above the middle
column. Then we stack them together to get Tower $n$, of height $h_n=3h_{n-1}+1$. We further denote this transformation, acting on the Borel probability
space $(X,\mathcal{A},\mu)$, as $T$.

Chacon's automorphism $T$ is known as an example of a weakly mixing but not strongly mixing transformation (proved by Chacon in \cite{mixing} for a historical version of the automorphism, constructed with division into two subcolumns instead of three, though
but the arguments also apply in the above-described case). It is proved by Del Junco
in \cite{deljunco} that T has a trivial centralizer. This result was improved by Del Junco, Rahe and Swanson in \cite{joinings}: they have shown that the transformation has minimal self-joinings. Prikhod'ko and Ryzhikov proved in \cite{convolutions} that the convolution powers of its maximal spectral type are pairwise mutually singular. They considered the Koopman operator $\hat{T}$ and the weak closure of its powers. The connection between Chacon's automorphism and an infinite family of polynomials over the rationals was first shown in this paper, where the polynomials as functions of $\hat{T}$ are identified to prove the statement.

For a given $\alpha \in \left[0, 1\right]$, an automorphism $S$ is said to be \emph{$\alpha$-weakly mixing} (for some $0\le \alpha \le 1$) if there exists an integer sequence $(m_j)$ such that $\hat{S}^{m_j}$ converges weakly to $\alpha \Theta + (1-\alpha)\mathrm{Id}$, where $\Theta$ is the ortho-projector to constants. The disjointness of the convolution powers follows from $\alpha$-weakly mixing property when $0 < \alpha < 1$ (see \cite{katok} and \cite{stepin}). This property is widely used in ergodic theory to build counterexamples (\cite{deljunco1992}). The question of $\alpha$-weak mixing for Chacon’s automorphism can be considered as a special case of the general problem of description of the weak limits of powers of $\hat{T}$. A similar method with application  of the Koopman operator and the weak
limits of its powers is used in \cite{delarue2012}. For examples of transformations
with non-trivial explicit weak closure of powers, see \cite{ryzhikov}.

E. Janvresse, A. A. Prikhod’ko, T. de la Rue, and V. V. Ryzhikov have shown in \cite{weaklimits} that the weak closure of the powers of $\hat{T}$ is reduced to $\Theta$
and an explicit family of polynomials $P_m(\hat{T})$:

$$
\mathcal{L} = {\Theta} \cup \{P_{m_1}(\hat{T}),\ldots,P_{m_r}(\hat{T})\hat{T}^n,r\ge0,\ 1\le m_1\le\ldots\le m_r,\ n\in \mathbb{Z}\}.
$$
 
 These polynomials are obtained from the representation of $T$ as an integral automorphism over the $3$-adic odometer (this representation was also used earlier in \cite{convolutions}). Several properties and the recurrence formulae for these polynomials are also described in \cite{weaklimits}.

The purpose of the present paper is to generalize the results of \cite{weaklimits} connected to the family of polynomials. We consider a parametric set of transformations similar to the Chacon’s automorphism and infer the properties and the recurrence equations for the polynomials generated by these transformations. The constructions are parametrized by an integer $p \ge 3$, and the classic case of Chacon's automorphism corresponds to the parameter value $p=3$. In \Cref{definitions}, we give a definition of the functional $\phi$, the polynomials $P_m^p(t)$ and the generalized automorphism $T$. We provide
inductive formulas for these polynomials in \Cref{recurrence}. The proofs of several important properties of $P_m^p(t)$, such as the palindromic property and the sequence of degrees, are provided in \Cref{palindromic} and \Cref{degrees} respectively. 

\section{Basic definitions}\label{definitions}
\subsection{The polynomials $P_m^p$ in the p-adic group}
Let $p \ge 3$ be an arbitrary integer.
Consider the compact group of $p$-adic numbers
$$\Gamma := \left\{x = \left(x_0, x_1, x_2, \ldots \right), x_k \in \{0, 1, \ldots, p - 1\} \right\}$$
We will also use the set
$$\Gamma' := \Gamma \setminus \{(p-1,p-1,\ldots)\},$$
where each element has only a finite number of leading elements equal to $(p-1)$.

Let $\lambda$ be the Haar measure on $\Gamma$. Under $\lambda$, the coordinates $x_k$ are i.i.d., uniformly distributed in $\{0, 1, \ldots, p - 1\}$.

    We define two measure-preserving (in terms of $\lambda$) transformations on $\Gamma'$:
    \begin{itemize}
    \item The shift-map $\sigma:\ x=\left(x_0, x_1, x_2, \ldots \right) \mapsto \sigma x = \left(x_1, x_2, \ldots \right)$
    \item The adding map $S:\ x \mapsto x + 1$, where $1:=(1,0,0,\ldots) \in \Gamma'$. (In general, each integer $j$ is
identified with an element of $\Gamma'$, so that $S^jx = x + j$ for all $j \in \mathbb{Z}$ and all
$x \in \Gamma'$.)
    \end{itemize}
    Let $\phi: \Gamma' \rightarrow \mathbb{Z}$ be the <<first not $(p-1)$>> functional:\\
    \[\phi(x):=x_i\text{ if }x=(p-1, \ldots, p-1, x_i, x_{i+1}, \ldots)\]
    \[\phi^{(0)}(x):=0;\ \phi^{(m)}(x):=\phi(x)+\phi(Sx)+\ldots+\phi(S^{m-1}x)\]

Let us define $\pi_m$ as the probability distribution of $\phi^{(m)}$ on $\mathbb{Z}$:
    \[\pi_m(j) = \lambda(\phi^{(m)}=j)\]
    We consider the sequences of polynomials $P_m^p$ produced by $\pi_m$ for fixed $p$ and $m$:
    \[P_m^p(t):= \mathbb{E}_\lambda\left[ t^{\phi^{(m)}(x)}\right] = \sum\limits_{j=0}^m \pi_m(j) t^j \]
    
\subsection{Integral automorphisms over the p-adic odometer}
Fix an arbitary $p \ge 3$. Let $\{h_n\}_{n \ge 0}$ be the sequence of heights: $h_0 := 1,\ h_n = ph_{n-1}+1$.

For each $n \ge 0$, we define

$$X_n:=\{(x,i): x \in \Gamma',\ 0 \le i \le h_{n-1} + \phi(x)\}$$
We consider the transformation $T_n$ of $X_n$, defined by
$$T_n(x, i) := \begin{cases}
(x,i+1), & \text{ if } i+1 \le h_n - 1 + \phi(x) \\
(Sx,0), & \text{ if } i=h_n-1+\phi(x). \end{cases}$$

The bijective map $\psi_n : X_n \mapsto X_n+1$, defined by
$$\psi_n(x,i):=(\sigma x, x_0 h_n + i + \mathbbm{1}\{x_0=p-1\}).$$
Observe that $\psi_n$, it conjugates the transformations $T_n$ and
$T_{n+1}$. We introduce the probability measure $\mu_n$ on $X_n$: for a fixed $i$ and a subset
$A \subset \{(x, i), x\in\Gamma' \}$,
$$\mu_n(A):=\frac{1}{h_n + 1/2} \lambda (\{x \in \Gamma', (x, i) \in A\}).$$

The transformation $T_n$ preserves the measure $\mu_n$, and the map $\psi_n$ sets the correspondence between $\mu_n$ and $\mu_{n+1}$. Thus, all the measure-preserving dynamical systems $(X_n, T_n, \mu_n)$ are isomorphic.

For $ 0 \le i \le h_{n-1}$, we set
$E_{n,i} := \{(x, i) : x \in \Gamma'\} \subset X_n$. We have $E_{n,i} = T^i_n E_{n,0}$, hence
 $$\{E_{n,0},\ldots,E_{n,h_n-1}\}$$
is a Rokhlin tower of height $h_n$ for $T_n$. Yet, for any $n \ge 0$, and any $0 \le i \le
h_n-1$,
$$\psi_n(En,i) = E_{n+1,i} \sqcup E_{n+1,h_n+i} \sqcup E_{n+1,2h_n+i+1}.$$
Fix $n_0$. A composition of the isomorphisms $(\psi_n)$, lets us observe all these Rokhlin
towers inside $X_{n_0}$. It follows from the formula above that the towers are embedded in the
same way as the towers of Chacon’s automorphism desribed in the introduction. In case $p=3$, it literally defines the Chacon's automorphism: $(X_{n_0} , \mu_{n_0} , T_{n0})$ is isomorphic to $(X,\mu,T)$. For an arbitrary $p > 3$, it generates a similar dynamical system which we call the \emph{generalized Chacon automorphism}. In this paper we discuss the properties of the generalized automorphism and infer how the known properties of the classical Chacon's automorphism change in the general case.

\section{Palindromic property}\label{palindromic}
In this section, we are going to prove that, just like in the classical case, the polynomials $P_m^p$ are palindromic for any parameter $p$. Moreover, we generalize the definition of $\phi$ functional and describe the class of Chacon-like functionals which produce palindromic polynomials.

Let $p \ge 3$ be an arbitrary integer.

We can generalize the definition of $\phi$ functional:
\begin{definition}
    $$
    \phi(x) = \begin{cases}
                    \omega(x_0), & 0 \le x_0 \le p - 2 \\
                    \phi(\sigma x), & x_0 = p - 1
                \end{cases}
    $$
\end{definition}
    The classic case is hence described by $\omega(j)=j$. This is the only non-trivial $\omega$ for $p=3$. \\
    
\begin{definition}
Let $\phi_\star = \min\limits_x \phi(x), \phi^\star = \max\limits_x \phi(x), \delta=\phi^\star - \phi_\star$. We say that a functional $\phi$  \textit{has the palindromic property} iff 
$$
\forall j,\ 0 \le j \le m\delta:   \pi_m(m\phi_\star + j)=\pi_m(m\phi^\star-j)
$$
\end{definition}
    
    Let us find the sufficient conditions for $\phi$ to have the palindromic property.
    
\begin{definition}
We call a function $\omega$  \textit{antipalindromic} iff
\[\begin{cases}
	\mathrm{Ran }\ \omega = \{0,1, \ldots, \zeta\} \\
	\forall j,\ 0 \le j \le p-2: \  \omega(j) = \zeta - \omega(p-2-j)
\end{cases}\]
\end{definition}

\begin{notation}
We further use $[M, N]$ instead of $\{M,M+1, \ldots, N\}$.
\end{notation}

Note that given $\omega$ such that $\mathrm{Ran}\ \omega = [0,\zeta]$, we have $\mathrm{Ran}\ \phi = [0,\zeta]$ and $\mathrm{Ran}\ \phi^{(m)} = [0,m\zeta]$

\begin{example}
The usual $\omega(j)=j$ is antipalindromic for $p = 3$. Indeed,
$$
\omega(0) = 0 = 1 - \omega(1) = 1 - \omega(3 - 2 - 0)
$$
$$
\omega(1) = 1 = 1 - \omega(0) = 1 - \omega(3 - 2 - 1)
$$
In the same way it is shown that $\omega(j)=j$ is antipalindromic for any $p \ge 3$.
\end{example}

\begin{theorem}
If $\omega$ is antipalindromic, the corresponding $\phi$ has the palindromic property.
\end{theorem}

We will need the following lemma to prove the theorem.

\begin{lemma}  \label{distLemma}
If $\omega$ is antipalindromic, the probability distributions of the random sequences $\{\phi(S^j x)\}_{j \ge 0}$ and $\{ \zeta - \phi(S^{-j} x)\}_{j \ge 0}$ are the same.
\end{lemma}
\begin{proof}
Let $x \in \Gamma \setminus \{(p-1,p-1,\ldots)\}$. We
say that $\mathrm{order}(x) = k \ge 0$ if $x_0=x_1=\ldots=x_{k-1}=p-1$ and $x_k \ne p-1$.
Since the first digit in the sequence $(\ldots, x-1,x,x+1)$ follows a periodic pattern
$(\ldots, 0, 1, 2, \ldots,p-2, p-1, 0, 1, 2, \ldots,p-2, p-1, \ldots)$, the contribution of points of order $0$ in the sequence
$\{\phi(S^j x)\}_{j \ge 0}$ provides a sequence of blocks $\omega(0),\omega(1),\ldots, \omega(p-2)$ separated by one symbol
given by a point of higher order. To fill in the missing symbols corresponding to positions $j$ 
such that $\mathrm{order}(x+j) \ge 1$, we observe that, if $x$ starts with a $p-1$, then for all $j \ge 0$, $\phi(x+pj)=\phi(\sigma x + j)$. Hence the missing symbols are given by the symbols $\{\phi(S^j \sigma x)\}_{j \ge 0}$.\\
Let us observe an example for $p=7$ produced by the Legendre symbol-like $\omega(x)=((\frac{x+1}{7}) + 1)/2$ (which is antipalindromic due to the Quadratic reciprocity):\bigskip\\
$\begin{array}{cccccccccccccccccccccccccccl}
1 & 1 & 0 & 1 & 0 & 0 & . & 1 & 1 & 0 & 1 & 0 & 0 &. & 1 & 1 & 0 & 1 & 0 & 0 & . & 1 & 1 & 0 & 1 & 0 & 0 &\leftarrow \text{contribution of order }0 \\
  &   &   &   &   &   & 1 &   &   &   &   &   &   &1 &   &   &   &   &   &   & 0 &   &   &   &   &   &   &\leftarrow \text{contribution of order }1 \\
1 & 1 & 0 & 1 & 0 & 0 & 1 & 1 & 1 & 0 & 1 & 0 & 0 & 1 & 1 & 1 & 0 & 1 & 0 & 0 & 0 & 1 & 1 & 0 & 1 & 0 & 0 &\leftarrow \text{the whole sequence} \\
\end{array}$
\bigskip\\
We may build $\{ \zeta - \phi(S^{-j} x)\}_{j \ge 0}$ from $\{\phi(S^j x)\}_{j \ge 0}$ with a composition of two measure-preserving transformations:
\begin{enumerate}
\item The 'reverse' transformation  $\{a_j\}_{j \ge 0} \mapsto \{a_{-j}\}_{j \ge 0}$ 
\item The 'flip' transformation which substitutes each element $k \in \{b_j\}_{j \ge 0}$ with $\zeta-k$, turning the sequence into $\{ \zeta - b_j\}_{j \ge 0}$ .
\end{enumerate}
By the definition of antipalindromic functions, it is clear that this procedure works as identity transformation when applied to $\{\phi(S^j x)\}_{j \ge 0}$.
\end{proof}

\begin{theorem}\label{pal_th}
If $\omega$ is antipalindromic, the corresponding $\phi$ has the palindromic property.
\end{theorem}
\begin{proof}
In Definition 1, we may substitute $\phi_\star = 0, \phi^\star = \delta = \zeta$. Hence it is enough to show that \[\forall j,\ 0 \le j \le m\zeta: \pi_m(j) = \pi_m(m \zeta - j).\]
$\pi_m(m \zeta - j) = \lambda(\phi^{(m)}(x)=m\zeta-j)=
\\=\sum\limits_{(\phi_1, \ldots, \phi_m)} \lambda\big(\phi(x)=\phi_1, \phi^{(2)}(x) = \phi_2, \ldots, \phi^{(m)}(x)=\phi_m\big) \mathbb{I} (\phi_1 + \ldots + \phi_k = m\zeta - j) = 
\\ = \sum\limits_{(\phi_1, \ldots, \phi_m)} \lambda\big(\phi(x)=\phi_1, \phi^{(2)}(x) = \phi_2, \ldots, \phi^{(m)}(x)=\phi_m\big) \mathbb{I} \big((\zeta-\phi_1) + \ldots + (\zeta-\phi_m) = j\big)$\\
Using \Cref{distLemma}, this equals to:\\
$\sum\limits_{(\phi_1, \ldots, \phi_m)} \lambda\big(\phi(x)=\zeta-\phi_m, \phi^{(2)}(x) = \zeta-\phi_{m-1}, \ldots, \phi^{(m)}(x)=\zeta-\phi_1\big) \mathbb{I} \big(\sum\limits_k (\zeta-\phi_k) = j\big) = \left< \psi_k := \zeta - \phi_k \right> = 
\\ =\sum\limits_{(\psi_1, \ldots, \psi_m)} \lambda(\phi(x)=\psi_1, \phi^{(2)}(x) = \psi_2, \ldots, \phi^{(m)}(x)=\psi_m) \mathbb{I} (\psi_1 + \ldots + \psi_m = j) =  \lambda(\phi^{(m)}(x)=j)= \\ =\pi_m(j)$
\end{proof}
We may describe a larger set of functionals $\phi$ having the palindromic property with the use of following lemma.
\begin{lemma}[On the affine transformations]\label{affineLemma}
Let $\omega:[0,p-2] \rightarrow [0, \zeta]$ be antipalindromic, then for any $a > 0, b \ge 0$:\\
$\phi'(x) = \begin{cases}
                    a \omega(x_0) + b, & 0 \le x_0 \le p - 2 \\
                    \phi'(\sigma x), & x_0 = p - 1
                \end{cases}$ is antipalindromic.
\end{lemma}
\begin{proof}
Let $\pi_m'(j) = \lambda(\phi'^{(m)}(x)=j)$. In terms of Definition 1, $\phi_\star = b, \phi^\star = a\zeta + b, \delta = a\zeta$. Let us prove that $\pi_m'(mb + j) = \pi_m' (m(a\zeta+b)-j)$.

First, we perform the division with remainder: $j = qa + r$. It follows from the construction of $\phi'$ that $\pi_m'(mb + qa + r) = 0$ if $r \ne 0$. Yet, $m(a\zeta+b)-j = m(a\zeta+b)-qa - r = mb + (m\zeta -q)a - r$ and hence $\pi_m' (m(a\zeta+b)-j) = 0$ if $r \ne 0$.\\
Thus, it remains to prove that $\pi_m'(mb + qa) = \pi_m' (m(a\zeta+b)-qa)$.
Note that we may restore the values of $\phi$ produced by $\omega$ from the values of $\phi'$. Indeed, consider the bijection $i: \{b, a + b, 2a + b \ldots, a\zeta + b\} \rightarrow [0, \zeta]$ such that $i(j) = \frac{j-b}{a}$. It's easy to see that $i(\phi'(x))=\phi(x)$. Subsequently, we may define $i^{(m)}(j) = \frac{j-mb}{m}$ and conclude $i^{(m)}(\phi'^{(m)}(x))=\phi(x)$.

Now let us prove $\pi_m'(mb + qa) = \pi_m' (m(a\zeta+b)-qa)$ using the fact that $i^{(m)}$ is bijective.

\[\pi_m'(mb + qa) = \lambda(\phi'^{(m)}(x)=mb + qa) = \lambda\big(i^{(m)}(\phi'^{(m)}(x))=i^{(m)}(mb + qa)\big)  =\]\[= \lambda(\phi^{(m)}(x)= \frac{mb-qa-mb}{m})= \lambda(\phi^{(m)}(x)= q)=\pi_m(q)\]
Similarly, $\pi_m' (m(a\zeta+b)-qa) = \pi_m(m\zeta - q)$. Since $\omega$ is antipalindromic, it follows from Theorem 2 that $\pi_m(q) = \pi_m(m \zeta - q)$ and then $\pi_m'(mb + qa) = \pi_m' (m(a\zeta+b)-qa)$.
\end{proof}
\begin{proposition}[On inheritance of palindromic property]
Let $\phi$ have the palindromic property, and let there be $\phi'$ such that for any $m \in \mathbb{N}$ there exists bijection $i^{(m)}: \mathrm{Ran}\ \phi'^{(m)} \rightarrow \mathrm {Ran}\ \phi^{(m)}$. Then $\phi'$ has the palindromic property.
\end{proposition}
The proof of this is the same as for \Cref{affineLemma}.

\begin{theorem}\label{pal_property}
The polynomials $P_m^p$ produced by the generalized Chacon automorphism have the palindromic property for any $p \ge 3$.
\end{theorem}
\begin{proof}
We have already shown that $\omega(j)=j$ is antipalindromic. Hence this theorem is a direct corollary of  \Cref{pal_th}.
\end{proof}

\section{Recurrence formulae for $P_m^p(t)$}\label{recurrence}
\begin{notation} We denote $n$th triangle number $\frac{n(n+1)}{2}$ as $\Delta_n$ \end{notation}
\begin{lemma}\label{simpleCase} $P_{pm}^p(t)=t^{m\Delta_{p-2}}P_m^p(t)$
\end{lemma}
\begin{proof}
Let $x \in \Gamma'$. Recalling the structure of $\{\phi(S^j x)\}_{j \ge 0}$, the value of $\phi^{(pm)}(x)$ is the sum of:
\begin{itemize}
\item the order-$0$ points. There are exactly $(p-1)m$ points following the repeating pattern $\ldots,(p-2),0,1,2,...,(p-2),0,1,\ldots$. Their contribution is $m$ times the sum of integers $0,1,\ldots,(p-2)$ which is $\frac{(p-1)(p-2)}{2}m=m\Delta_{p-2}$
\item the higher-order points. Their contribution is $\phi^{(m)}(\sigma x)$.
\end{itemize}
Hence $\phi^{pm}(x) = m\Delta_{p-2} + \phi^{(m)}(\sigma x)$. By the definition of polynomials $P_m(t)$ it implies 
$P_pm(t) = \mathbb{E}_\lambda\left[ t^{\phi^{(pm)}(x)}\right] = 
\mathbb{E}_\lambda\left[ t^{m\Delta_{p-2} + \phi^{(m)}(\sigma x)}\right] = 
t^{m\Delta_{p-2}} \mathbb{E}_\lambda \left[ t^{\phi^{(m)} (\sigma x)}  \right] = 
t^{m\Delta_{p-2} }P_m(t)$
\end{proof}
\begin{notation}
Similarly to $\phi^{(m)}$, we denote $\omega^{(m)}(j) := \omega(j) + \omega(j+1) + \ldots + \omega(j+m-1)$ if $j + m < p - 1$.
\end{notation}
\begin{proposition} $\omega^{(k)}(j) = k j + \Delta_{k-1}$
\end{proposition}
\begin{proof}
The value of $\omega^{(k)}(x_0)$ is the sum of an arithmetic progression $x_0, x_0 + 1, \ldots, x_0 + k-1$.
\end{proof}
\begin{lemma} \label{phiLemma}
Let $0 < k < p$. Consider $x=(x_0,x_1,\ldots) \in \Gamma'$.

$\phi^{(pm+k)}(x)= 
\begin{cases}
\omega^{(k)}(x_0) + m\Delta_{p-2} + \phi^{(m)}(\sigma x), & x_0 < p-k\\
\omega^{(p-1-x_0)}(x_0) + \Delta_{x_0+k-p-1}+m\Delta_{p-2}+\phi^{(m+1)}(\sigma x), & x_0 \ge p-k
\end{cases}$
\end{lemma}
\begin{proof}
First, by the definition of $\phi^(k)$ we state $$\phi^{(pm+k)}(x) = \phi(x) + \phi(Sx) + \ldots + \phi(S^{k-1}x) + \phi^{(pm)}(S^k x)= \phi^{(k)}(x) + \phi^{(pm)}(S^k x).$$
With the use of the previous lemma holds the equality $\phi^{(pm+k)}(x) =  \phi^{(k)}(x) + m\Delta_{p-2} + \phi^{(m)}(\sigma S^k x)$.
By the definition of $\phi(x)$ there are two cases in the computation of $\phi^{(pm+k)}(x)$:
\begin{itemize}
\item The regular case $x_0 < p-k $. In this case each term in $\phi^{(k)}(x)$ is computed directly: $\phi^{(k)}(x) = \omega^{(k)}(x)$. Yet $\sigma S^k x = \sigma x$: since $S^k$ affects only the first digit of $x$, its effect if erased from $\sigma S^k x$. We may compute $\phi^{(pm+k)}(x)=\omega^{(k)}(x_0)+m\Delta_{p-2}+\phi^{(m)}(\sigma x)$.
\item In another case, if $x_0 \ge p-k $, some term  in $\phi^{(k)}(x)$ evaluates with recursion: there exists $j \in [0,k-1]$ such that $S^j x$ begins with $p-1$ and hence $\phi(S^j x) = \phi(\sigma x)$. Therefore we may not compute $\phi^{(k)}(x)$ directly. Instead, we divide it into $\phi^{(p-1-x_0)}(x)=\omega^{(p-1-x_0)}(x)$, which is computed directly, and the rest of the terms. The latter form the sum $Z := \phi(S^j x) + \phi(S^{j+1}x) + \ldots + \phi(S^{k-1})$. Consider this sum. As stated before, $S^j x$ begins with $p-1$, so the first digits of $S^{j+1}x, S^{j+2}x, \ldots, S^{k-1}$ are $0,1,\ldots, x_0+k-p-1$. We know $\phi(S^j x) = \phi(\sigma x)$, and, knowing the first digits of the following terms, we may evaluate $Z = \phi(\sigma x) + 0 + 1 + \ldots + (x_0+k-p-1) =  \phi(\sigma x) + \Delta_{x_0+k-p-1}$. Now we aggregate $\phi^{(k)}(x)=\omega^{(p-1-x_0)}(x) + \Delta_{x_0+k-p-1}+\phi(\sigma x)$. Using the equality $\phi^{(m)}(\sigma S^k x) = \phi^{(m)}(S\sigma x) = \phi^{(m+1)}(\sigma x) - \phi(\sigma x)$, we get 
$$\phi^{(pm+k)}(x) =  \phi^{(k)}(x) + m\Delta_{p-2} + \phi^{(m)}(\sigma S^k x) = $$ $$ = \omega^{(p-1-x_0)}(x) + \Delta_{x_0+k-p-1}+\phi(\sigma x)+ m\Delta_{p-2} + \phi^{(m+1)}(\sigma x) - \phi(\sigma x) = $$ $$ = \omega^{(p-1-x_0)}(x_0) + \Delta_{x_0+k-p-1}+m\Delta_{p-2}+\phi^{(m+1)}(\sigma x).$$
\end{itemize}
\end{proof}
\begin{lemma} \label{complexCase}
Let $0 < k < p$.

$$P_{pm+k}^p(t)  =   \frac{1}{p} t^{m\Delta_{p-2} + \Delta_{k-1}}\sum\limits_{j=0}^{p-k-1} t^{jk} P_m^p(t) 
      + \frac{1}{p} t^{m\Delta_{p-2} + \Delta_{k-2}}\sum\limits_{j=0}^{k-1} t^{j(p-k)} P_{m+1}^p(t) $$
\end{lemma}
\begin{proof}
This lemma is shown using the previous result and the law of total expectation.
$$ P_{pm+k}(t)= \mathbb{E}_\lambda\left[ t^{\phi^{(pm+k)}(x)}\right] = $$
$$= \mathbb{E}_\lambda\left[ t^{\phi^{(pm+k)}(x)}| x_0 < p-k \right]\lambda(x_0<p-k) + \mathbb{E}_\lambda\left[ t^{\phi^{(pm+k)}(x)} | x_0 \ge p-k \right]\lambda( x_0 \ge p-k )$$
\begin{equation}\label{eq:totalExp}= \mathbb{E}_\lambda\left[ t^{\phi^{(pm+k)}(x)}| x_0 < p-k \right]\frac{p-k}{p} + \mathbb{E}_\lambda\left[ t^{\phi^{(pm+k)}(x)} | x_0 \ge p-k \right]\frac{k}{p}
\end{equation}
Now we calculate each conditional expectations according to \Cref{phiLemma}. Note that the digits of $x$ are mutually independent, hence $x_0 \indep (x_1, x_2,\ldots)$ and then $x_0 \indep \sigma x$, which is used in the evaluation.
$$\mathbb{E}_\lambda\left[ t^{\phi^{(pm+k)}(x)}| x_0 < p-k \right] = 
\mathbb{E}_\lambda\left[ t^{ \omega^{(k)}(x_0) + m\Delta_{p-2} + \phi^{(m)}(\sigma x) }| x_0 < p-k \right] = $$ 
$$ =t^{m\Delta_{p-2}}
\mathbb{E}_\lambda\left[ t^{ \omega^{(k)}(x_0)} | x_0 < p-k \right]
 \mathbb{E}_\lambda\left[ t^{\phi^{(m)}(\sigma x) } \right]
$$
The distributions of $\sigma x$ and $x$ are the same, which implies
$\mathbb{E}_\lambda\left[ t^{\phi^{(m)}(\sigma x) } \right] = \mathbb{E}_\lambda\left[ t^{\phi^{(m)}(x) }\right] = P_m^p(t)$. It remains to evaluate $\mathbb{E}_\lambda\left[ t^{ \omega^{(k)}(x_0) | x_0 < p-k }\right]$:
$$ \mathbb{E}_\lambda\left[ t^{ \omega^{(k)}(x_0)} | x_0 < p-k \right] =
 \frac{1}{p-k}\sum\limits_{j=0}^{p-k-1} t^{ \omega^{(k)}(j)} = 
 \frac{1}{p-k}\sum\limits_{j=0}^{p-k-1} t^{kj + \Delta_{k-1}} =
 \frac{1}{p-k} t^{\Delta_{k-1}}\sum\limits_{j=0}^{p-k-1} t^{kj}$$
 Thus, \begin{equation}\label{eq:firstCondExp}\mathbb{E}_\lambda\left[ t^{\phi^{(pm+k)}(x)}| x_0 < p-k \right] =
  \frac{1}{p-k} t^{m\Delta_{p-2} + \Delta_{k-1}}\sum\limits_{j=0}^{p-k-1} t^{jk} P_m^p(t)
   \end{equation}
  
   Similarly, we evaluate the second conditional expectation. 
 $$\mathbb{E}_\lambda\left[ t^{\phi^{(pm+k)}(x)}| x_0 \ge p-k \right] = 
 \mathbb{E}_\lambda\left[ t^{\omega^{(p-1-x_0)}(x_0) + \Delta_{x_0+k-p-1}+m\Delta_{p-2}+\phi^{(m+1)}(\sigma x)}| x_0 \ge p-k \right] = $$
 $$ =t^{m\Delta_{p-2}}
 \mathbb{E}_\lambda\left[ t^{\omega^{(p-1-x_0)}(x_0) + \Delta_{x_0+k-p-1}} | x_0 \ge p-k \right] 
 \mathbb{E}_\lambda\left[ t^{\phi^{(m+1)}(\sigma x) } \right] = $$
  $$ =t^{m\Delta_{p-2}}
    P_{m+1}^p(t)
 \mathbb{E}_\lambda\left[ t^{\omega^{(p-1-x_0)}(x_0) + \Delta_{x_0+k-p-1}} | x_0 \ge p-k \right] = $$
  $$ = \frac{1}{k} t^{m\Delta_{p-2}} P_{m+1}^p(t)
  \sum\limits_{j=p-k}^{p-1} t^{  (p-1-j) j + \Delta_{p-2-j} + \Delta_{j+k-p-1}}  = $$

We substitute the index $j$ with $q = p-1-j$ and simplify the exponent:
$$  (p-1-j) j + \Delta_{p-2-j} + \Delta_{j+k-p-1} =
    q(p-1-q) + \Delta_{q-1} + \Delta_{k-2-q}=
   q(p-k)+\Delta_{k-2} 
$$   
Now we collect the result: 
\begin{equation}\label{eq:secondCondExp}
\mathbb{E}_\lambda\left[ t^{\phi^{(pm+k)}(x)}| x_0 \ge p-k \right] 
 =  \frac{1}{k} t^{m\Delta_{p-2} + \Delta_{k-2}} P_{m+1}^p(t)
  \sum\limits_{q=0}^{k-1} t^{q(p-k)}
  \end{equation}
It remains to substitute (\ref{eq:firstCondExp}) and (\ref{eq:secondCondExp}) into (\ref{eq:totalExp}) to obtain the recurrence equation:
$$P_{pm+k}(t)  =   \frac{1}{p} t^{m\Delta_{p-2} + \Delta_{k-1}}\sum\limits_{j=0}^{p-k-1} t^{jk} P_m^p(t) 
      + \frac{1}{p} t^{m\Delta_{p-2} + \Delta_{k-2}}\sum\limits_{j=0}^{k-1} t^{j(p-k)} P_{m+1}^p(t) $$
\end{proof}
Bringing together the results of \Cref{simpleCase} and \Cref{complexCase}, we obtain the general recurrence law for $P_{n}^p(t)$
\begin{theorem}[Recurrence formulae on $P_{n}^p(t)$] 
\begin{align}
\label{P_pm} P_{pm}^p(t) & =  t^{m\Delta_{p-2}}\ P_m^p(t) \\
\label{P_pm+k} P_{pm+k}^p(t) & =  \frac{1}{p} t^{m\Delta_{p-2} + \Delta_{k-1}}\sum\limits_{j=0}^{p-k-1} t^{jk} P_m^p(t)  + \frac{1}{p} t^{m\Delta_{p-2} + \Delta_{k-2}}\sum\limits_{j=0}^{k-1} t^{j(p-k)} P_{m+1}^p(t), \ 0 < k < p
 \end{align}
\end{theorem}

\section{Degrees of polynomials}\label{degrees}
In this section, our goal is to describe a handy symmetric form of the polynomials $P_m^p(t)$. This requires a more detailed consideration of the degrees of these polynomials.
\subsection{Sequence $\hat{S}_j$}
Here we define the substitution-produced sequence $\{\hat{S}_j\}_{j \ge 0}$. It will be useful to evaluate the degrees of $P_m^p(t)$. We build this sequence recursively. First, fix the parameter value $p \ge 3$.
\begin{definition}[Zero-tier construction]
Let $\hat{S}_0^{(0)} = p-2,\ \hat{S}_j^{(0)} = p-2-(j-1)$ for $j = \overline{1, p-1}$.
\end{definition}
Literally, $\{ \hat{S}_j^{(0)}\}$ is $p-2, p-2, p-3, \ldots, 2, 1, 0$.
\begin{definition} [Number substitution rule]
We substitute
\begin{itemize}
 \item $0$ with $p-2, p-3, \ldots, 2, 1, 0, 0$;
 \item $1$ with $p-2, p-3, \ldots, 2, 1, 1, 0$;
 \item $2$ with $p-2, p-3, \ldots, 2, 2, 1, 0$;
 \item \ldots
 \item $p-3$ with $p-2, p-3, p-3, \ldots, 2, 1, 0$;
 \item $p-2$ with $p-2, p-2, p-3, \ldots, 2, 1, 0$;

 \item etc.
 \end{itemize}
 Each number $q = \overline{0, p-2}$ becomes the descendent sequence of numbers from $p-2$ to $0$ with $q$ taken twice in a row.
\end{definition}
\begin{definition}[First-tier construction]\label{tier1}
 We have built this sequence $\{ \hat{S}_j^{(0)}\}$  of $p$ integers. To define 
$\hat{S}_j^{(1)}$, we substitute each number in $\{ \hat{S}_j^{(0)}\}$ according to the rule. As a result, we have the sequence $\{\hat{S}_j^{(1)}\}$ of $p^2$ integers.
\end{definition}
We continue the procedure ad infinitum, substituting each number in $\{ \hat{S}_j^{(l)}\}$  to define $\{ \hat{S}_j^{(l+1)}\}$ . It is easy to see that
\begin{itemize}
\item on each step the length of the sequence multiplies by $p$ and
\item for each $l \ge 0$,  $\{ \hat{S}_j^{(l)}\}$ is a prefix of $\{ \hat{S}_j^{(l+1)}\}$.
\end{itemize}

There exists a countably infinite sequence $\{ \hat{S}_j^{(\infty)}\}$ which contains each of $\{ \hat{S}_j^{(l)}\}$ as a prefix.
\begin{definition}
Let $\{\hat{S}_j\}_{j \ge 0} := \{ \hat{S}_j^{(\infty)}\}_{j \ge 0}$.
\end{definition}

We could as well define the substitution rule first and then generate the sequence from initial datum $s_0 = p-2$, i. e. the sequence is is determined in a unique way by the rule and the initial element. This notion will be useful further.

\begin{example}
Let $p=4$.

$\begin{array}{|c|cccc|cccc|cccc|cccc|}
\hline
\{ \mathbf{\hat{S}_j^{(0)}}\} & 2 & & & & 2 & & & & 1 & & & & 0 & & & \\
\hline
\{ \mathbf{\hat{S}_j^{(1)}}\} & 2 & 2 & 1 & 0 & 2 & 2 & 1 & 0 & 2 & 1 & 1 & 0 & 2 & 1 & 0 & 0 \\
\hline
\{ \mathbf{\hat{S}_j^{(2)}}\} &
\text{\small{2210}} & \text{\small{2210}} & \text{\small{2110}} & \text{\small{2100}} &
\text{\small{2210}} & \text{\small{2210}} & \text{\small{2110}} & \text{\small{2100}} &
\text{\small{2210}} & \text{\small{2110}} & \text{\small{2110}} & \text{\small{2100}} &
\text{\small{2210}} & \text{\small{2110}} & \text{\small{2100}} & \text{\small{2100}} \\
\hline
 \mathbf{\ldots} & & & & & & & & & & & & & & & & \\
\hline
\end{array} $
\end{example}

\begin{proposition}[Characteristic property of $\{\hat{S}_j\}$]\label{charProp}
For each $m > 0,\ 0 \le j \le p-1$: $\hat{S}_{pm+j}=\hat{S}_{p(p-1-S_m) + j}$
\end{proposition}
\begin{proof}
First, we find the number of the first tier where the $(pm+j)$-th element appears: $l = \log_p{m}$. This $l$-th tier consists of $p$ blocks, each having $p^l$ elements, and each block represents an element from tier $l-1$. The element  $\hat{S}_{pm+j}$ belongs to the $m$-th block of tier $l$, which is the expansion of the $m$-th element in  tier $l-1$. Hence,  
$\hat{S}_{pm+j}$ is the $j$-th element in the substitution of $\hat{S}_m$.

Note that all the substitutions are listed in tier $1$, which consists of the elements from $\hat{S}_0$ to $\hat{S}_{p^2-1}$.
Rigorously, the substitution of $q$ if represented with the items from $\hat{S}_{p + p(p-2-q)}$ to $\hat{S}_{p + p(p-2-q) + p-1}$. If we desire the $j$-th element of this substitution (counting from 0), it is $\hat{S}_{p + p(p-2-q) + j}$. 

Using this result and simplifying the index, we show that the $j$-th element in the substitution of $\hat{S}_m$ is $\hat{S}_{p(p-1-S_m) + j}$. 
\end{proof}
\begin{example}
Let $p=4$. Our goal is to find $\hat{S}_{141}$.
$$ \hat{S}_{141} = \hat{S}_{4 \cdot 35 + 1} = \hat{S}_{4(3-\hat{S}_{35}) +1}$$
$$ \hat{S}_{35} = \hat{S}_{4 \cdot 8 + 3} = \hat{S}_{4(3-S_{8}) +3}$$
$$\hat{S}_8 = 2 \implies \hat{S}_{35}=\hat{S}_{4(3-2)+3}=\hat{S}_7=0$$
$$\hat{S}_{141}=\hat{S}_{4(3-\hat{S}_{35}) +1} = \hat{S}_{4(3-0)+1} = \hat{S}_{13}=1$$

The answer is 1.

\end{example}

With the use of this property, we may make a more handy definition of $\hat{S}$. Indeed, as shown in the example, for any $j \ge p^2$ we may find $\hat{S}_j$ applying the equation from \Cref{charProp} sufficient number of times. Hence it is enough to define only the tier 1 (i.e. the values of $\hat{S}_0, \ldots, \hat{S}_{p^2-1}$ to construct the whole sequence.
\begin{definition}\label{sm_good_def}
Define $\{\hat{S}_j^{(1)}\}$ as in \Cref{tier1}. Let
$$ \hat{S}_{pm+j} =
 \begin{cases}
     \hat{S}_{pm+j}^{(1)}, & m < p \\
     \hat{S}_{p(p-1-S_m) + j}, & m \ge p
 \end{cases}
$$
\end{definition}

We will use this definition further.

\subsection{Formulae for degrees}
\begin{definition}
Let $m \ge 0$. We define $D_m^p := \mathrm{deg}\ P_m^p$.
\end{definition}
\begin{proposition} \label{dm_eq}
For any $m \ge 0,\ 0 < k < p$:
\begin{align*}
  D_{pm}^p &= m \Delta_{p-2} + D_m^p \\
  D_{pm+k}^p &= m \Delta_{p-2}+(p-k)(k-1)+\Delta_{k-2}+D_m^p + \mathrm{max}\{p-k-1, D_{m+1}^p-D_m^p\}
\end{align*}
\end{proposition}
\begin{proof}
Directly from the (\ref{P_pm}) and (\ref{P_pm+k}) we obtain the equations for degrees:
\begin{align*}
  D_{pm}^p &= m \Delta_{p-2} + D_m^p \\
  D_{pm+k}^p &= \mathrm{max}\{ m\Delta_{p-2} + \Delta_{k-1} + k(p-k-1) + D_m^p,\     
                                m\Delta_{p-2} + \Delta_{k-2} + (p-k)(k-1) + D_{m+1}^p \}
\end{align*}
To show this proposition, it is enough to take the common terms out of the maximum operator.
\end{proof}
The following holds for arbitrary fixed $p \ge 3$. Keeping this in mind, we use $D_j$ instead of $D_j^p$ for short.
\begin{definition} \label{sm_def}
For any $m \ge 0$, we define $S_m := D_{m+1}^p - D_m^p $.
\end{definition}
Our goal is to show that this sequence $\{S_m\}_{m \ge 0}$ is equal to $\{\hat{S}_m\}$ defined in the previous Subsection. First, we deduce recurrence equations on $S_m$.
\begin{proposition}
\begin{equation} \label{sm_eq}
S_{pm+k} = p-k-1 - \mathbbm{1}\{S_m < p-1-k\}
\end{equation}
\end{proposition}
\begin{proof}
We will show this using \Cref{dm_eq} and \Cref{sm_def}. If we try to evaluate $S_j$, it collapses into two cases:
\begin{itemize}
\item Let $k<p-1 $. $$S_{pm+k} = D_{pm+k+1}-D_{pm+k} =$$ $$= m\Delta_{p-2} + (p-k-1)k + \Delta_{k-1} + D_m + \max\{p-k-1, D_{m+1}-D_m\} - $$ $$ - m\Delta_{p-2}-(p-k)(k-1)-\Delta_{k-2}-D_m-\max\{p-k,D_{m+1}-D_m\}=$$ $$ =p-k-1+\max\{p-k-1, D_{m+1}-D_m\}-\max\{p-k,D_{m+1}-D_m\};$$
We substitute $D_{m+1}-D_m = S_m$ and observe that $$\max\{p-k-1, D_{m+1}-D_m\}-\max\{p-k,D_{m+1}-D_m\}= \mathbbm{1}\{S_m < 1-k\}.$$ This substitution leads to exactly the equation in the proposition.
\item $S_{pm+p-1} = D_{p(m+1)}-D_{pm+3} = \\=(m+1)\Delta_{p-2}+D_{m+1}-m\Delta_{p-2}-(p-(p-1))(p-1-1)-\Delta_{p-1-2}-D_m-\max\{0, D_{m+1}-D_m\}=\\=\Delta_{p-2}+(D_{m+1}-D_m)-(p-2)-\Delta_{p-3}-(D_{m+1}-D_m)=0$
\end{itemize}

A straightforward substitution shows that the second case is governed by the same equation.
\end{proof}

\begin{lemma}\label{sm1}
For $j = 0, 1, \ldots, p^2-1:\ S_j=\hat{S}_j^{(1)}$
\end{lemma}
\begin{proof}
We compute the initial $p^2$ terms of this sequence. Recall that $P_0^p(t)=1,\ P_1^p(t) = \frac{1}{p-1} \sum\limits_{j=0}^{p-2} t^j$, hence $D_0 = 0,\ D_1 = p-2$ and $S_0 = p-2$. We will use this as the initial condition in the recurrence equation from (\ref{sm_eq}).
Observe the equation (\ref{sm_eq}). The right part consists of a periodic term $p-k-1$ with period $p$ and the step function $-\mathbbm{1}\{S_m < p-1-k\}$ which is determined by the number of period. Thus we evaluate the sequence $\{S_j\}$ period-wise.

In the first period, $\{S_{0p+k}\}$, we have already shown $S_0=p-2$. The step function is $-\mathbbm{1}\{k < 1\}$, so it affects none of the terms $S_1, \ldots, S_{p-1}$. They form the descending progression $p-2, p-3, \ldots, 1, 0$. Now we have the first $p$ terms of $\{S_j\}$ which are to be used to obtain the subsequent terms.

In the second period, the step function is $-\mathbbm{1}\{S_1 < p-1-k\}$. We know that $S_1=2$, thus the second period repeats the first. In the third period we have $-\mathbbm{1}\{S_1 < p-1-k\}=-\mathbbm{1}\{p-3 < p-1-k\} = -\mathbbm{1}\{k < 2\}$ which affects first two terms. The third period is $p-2,p-3,p-3,p-4, \ldots, 1, 0$.

Similarly, we use all of the terms in the first period to evaluate each of $S_j$ for $j < p^2$. 
It is shown with ease that they the same as in $\hat{S}_j^{(1)}$
\end{proof}

\begin{proposition} \label{sm_estimate}
$\forall j \ge 0:\ S_j \in \{0, 1, \ldots, p-2\}$
\end{proposition}
\begin{proof}
The explicit evaluation of $S_0, S_1, \ldots, S_{p-1}$ shows 
\begin{equation}\label{subset}
\{0,1,\ldots,p-2\} \subset \mathrm{Ran\ } \{S_j\}
\end{equation}

Obviously, $S_j \ge 0$.
We prove by induction that $\forall m \ge p,\ 0 < k < p-1:\ S_{pm+k} < p-1$. The order $\log_p m =: u $ is the parameter of induction. The explicit evaluation of $S_p, S_{p+1}, \ldots, S_{p^2-1}$ shows the base of induction. Let the statement be true for an arbitrary $u$. When computing the values of $S_j$ for order $u+1$ from (\ref{sm_eq}), we refer to the values of $S_m$ of order $u$. In (\ref{sm_eq}), $p-k-1 \le p-1$ and $- \mathbbm{1}\{S_m < p-1-k\} \le 0$, thus $S_j \le p-1$. Let us show that the equality is never reached. Indeed, $p-k-1 = p-1$ if true only for $k=0$. When we substitute $k=0$ into the step function, we get $- \mathbbm{1}\{S_m < p-1\}$. For $S_m$ of order $u$, the statement in brackets is true due to the induction hypothesis. So, in order $u+1$ the statement $S_j < p-1$ holds. The inductive step is now shown.

Now we may assert $\mathrm{Ran}\ \{S_j\} \subset \{0,1,\ldots,p-2\}$. Bringing together this and (\ref{subset}), we infer the proposition.
\end{proof}
\begin{lemma} \label{sm2}
For each $m \ge p,\ 0 < k < p-1:\ S_{pm+k}=S_{p(p-1-S_m)+k}$
\end{lemma}
\begin{proof}
First, we observe that for any $j \ge 0$, $S_j = S_{p-1 - S_j}$. Using the \Cref{sm_estimate}, it is enough to show the statement for $S_j=0,1,\ldots,p-2$. For example, let $S_j=0$. We make sure that $0 = S_{p-1}$ which we know from the direct computation. Similarly, we do this for any $S_j=1,\ldots,p-2$.

To prove the lemma, we will this equation. Let $m \ge p,\ 0 < k < p-1$. From (\ref{sm_eq}) we obtain $S_{pm+k}=p-k-1 - \mathbbm{1}\{S_m < p-1-k\}$ and $S_{p(p-1-S_m)+k}= p-k-1 - \mathbbm{1}\{S_{p-1-S_m} < p-1-k\}$. One may easily observe that the equality $S_m=p-1-S_m$ turns this into an identity.
\end{proof}
\begin{theorem} \label{subst_th}
For each $j \ge 0$: $S_j=\hat{S}_j$.
\end{theorem}
\begin{proof}
The results of \Cref{sm1} and \Cref{sm2} show that \Cref{sm_good_def} is true for $\{S_j\}_{j \ge 0}$.
\end{proof}
\begin{corollary}
$D_m = \sum\limits_{j=0}^{m-1} \hat{S}_j$
\end{corollary}
\subsection{Formulae for lower degrees}
\begin{definition}
Let $m \ge 0$. The \emph{lower degree} $d_m^p$ of the polynomial $P_m^p(t)$ is the smallest degree of $t$ among the terms of $P_m^p(t)$.
\end{definition}

In other words, $d_m^p$ is such an integer that $P_m^p(t)=t^{d_m^p}Q(t)$, where $Q(t)$ is a polynomial having a constant term.

In this subsection we will show that the lower degree $d_m^p$ follows a substitution pattern similar to the degree $D_m^p$.
\begin{proposition}
For any $m \ge 0,\ 0 < k < p$:
\begin{align*}
  d_{pm}^p &= m \Delta_{p-2} + d_m^p \\
  d_{pm+k}^p &= m \Delta_{p-2}+\Delta_{k-2}+d_m^p + \mathrm{min}\{k-1, d_{m+1}^p-d_m^p\}
\end{align*}
\end{proposition}
\begin{proof}
This is shown directly from the equations (\ref{P_pm}) and (\ref{P_pm+k}).
\end{proof}
\begin{definition} \label{lower_sm_def}
For any $m \ge 0$ and a fixed $p$, we define $s_m := s_{m+1}^p - s_m^p $.
\end{definition}

Similarly to $S_m$ from the previous subsection, $s_m$ is a substitution sequence. This fact is shown in complete analogy with \Cref{subst_th}, so we omit the details of the proof and just describe the substitution rule.

\begin{theorem} \label{lower_subst_th}
The sequence $s_m$ is generated by substitutions:
\begin{itemize}
\item $s_0$ = 0
\item each $s_j$ is substituted with $0, 1, \ldots, s_j-1, s_j, s_j, s_j+1, \ldots, p-2$ (which is the increasing sequence of integers from $0$ to $p-2$ with $s_j$ taken twice).
\end{itemize}
\end{theorem}
\begin{example}
We evaluate the first elements of $\{s_j\}_{j \ge 0}$ for $p=4$ with the substitution rule:
$0 \mapsto 0012 \mapsto 0012\ 0012\ 0112\ 0122 \mapsto [0012\ 0012\ 0112\ 0122]\ [0012\ 0012\ 0112\ 0122]\  [0012\ 0112\ 0112\ 0122]\ [0012\ 0112\ 0122\ 0122] \mapsto \ldots$
Writing the last expression in a row, we get the first $64$ elements of $\{s_j\}$.
\end{example}
\subsection{Symmetric forms of polynomials}
Here we use the results of \Cref{subst_th} and \Cref{lower_subst_th} to evaluate a reduced form of the recurrence equations (\ref{P_pm}) and (\ref{P_pm+k}).
\begin{lemma}
For each $j \ge 0$, $S_j + s_j = p-2 = const$
\end{lemma}
\begin{proof}
We prove this by induction. As a base, we recall that $S_j$ is generated from $S_0 = p-2$ and $s_j$ is generated from $s_0 = 0$. For index $0$, the statement holds: $S_0 + s_0 = p-2$.

Now we prove the induction step. Let $j$ be such that $S_j + s_j = p-2$. We state that when we substitute $S_j$ and $S_j$ with the proper sequences of length $p$, the statement holds for each pair of elements in these subsequences. This follows directly from the substitution rules.
\end{proof}

\begin{definition}
The \emph{mid-degree} $\delta_m^p$ of a polynomial $P_m^p(t)$ is the arithmetic mean of its degree and its lower degree: $\delta_m^p = \frac{D_m^p+d_m^p}{2}$
\end{definition}
\begin{corollary}
$\delta_m^p = \frac{(p-2)m}{2}$
\end{corollary}
\begin{proof}
For a fixed $p$, $\delta_m^p =  \frac{D_m^p+d_m^p}{2} =  \frac{\sum\limits_{j=0}^{m-1}S_j + \sum\limits_{j=0}^{m-1}s_j}{2} =  \frac{\sum\limits_{j=0}^{m-1}(S_j + s_j)}{2} = m \frac{p-2}{2}$
\end{proof}
\begin{definition}
We define the \emph{symmetric form} of polynomials $P_m^p(t)$:
$$ \SymPol{p}{m}(t) = t^{-\delta_m^p}P_m^p(t)$$
\end{definition}

These are polynomial-like functions of $t$. The terms in the symmetric form are monomials of $t$ in real degrees. If the number of terms in $P_m^p$ is odd, then the degrees of $t$ in $\SymPol{m}{p}(t)$ are integers, otherwise they are half-integers.

The degrees of $t$ in $\SymPol{m}{p}(t)$ take from $-\frac{D_m^p}{2} $ to $\frac{D_m^p}{2} $, half positive and half negative (and a term of degree 0, i.e. constant, if the number of terms is odd). 

\begin{theorem}(Recurrence equations on the symmetric forms of $P_m^p$)
\begin{align}
    \SymPol{p}{pm}(t)&=\SymPol{p}{m}(t)\\
    \SymPol{p}{pm+k}(t)&=\frac{1}{p}\left(\sum\limits_{j=-\frac{p-k-1}{2}}^{\frac{p-k-1}{2}} t^{jk}\right)\SymPol{p}{m}(t)+\frac{1}{p}\left(\sum\limits_{j=-\frac{k-1}{2}}^\frac{k-1}{2} t^{j(p-k)}\right)\SymPol{p}{m+1}(t)
\end{align}
\end{theorem}
In these sums, indexes may not be integers, which is not absolutely strict though compact. It is implied that the sum index $j$ increases by one, anyway. For example, given $p=5$ and $k=1$, the expression $\left(\sum\limits_{j=-\frac{p-k-1}{2}}^{\frac{p-k-1}{2}} t^{jk}\right)$ means $t^{-\frac{3}{2}} + t^{-\frac{1}{2}} + t^{\frac{1}{2}} +  t^{\frac{3}{2}}$.
\begin{proof} Substituting $P_m^p(t) = t^{\delta_m^p}\SymPol{m}{p}(t)$, we deduce these equations directly from (\ref{P_pm}) and (\ref{P_pm+k}).
\end{proof}
\newpage
\bibliography{chacon}

\end{document}